\documentclass[12pt,a4paper]{article}

\usepackage{amsthm}
\usepackage{graphics,epsfig,color}

\newcommand{\Int}{\mathop{\mathrm{Int}}}

\def\bord{\partial}

\let\bydef\emph
\let\implies\Longrightarrow

\newtheorem{theo}{Theorem}[section]
\newtheorem{corol}[theo]{Corollary}
\newtheorem{prop}[theo]{Proposition}
\newtheorem{lem}[theo]{Lemma}

\theoremstyle{definition}
\newtheorem*{defi}{Definition}

\newtheorem*{claim}{Claim}

\def\rr{\mathbf{R}}

\def\zz{\mathbf{Z}}
\def\nn{\mathbf{N}}

\def\hh{\mathbf{H}}

\def\calF{\mathcal{F}}
\def\calS{\mathcal{S}}
\def\calT{\mathcal{T}}

\def\calV{\mathcal{V}}

\def\oo{\mathcal{O}}

\title{A structure theorem for irreducible open graph 3-manifolds}
\author{Sylvain Maillot}

\begin{document}

\maketitle

\begin{abstract}
Graph manifolds are a class of compact, orientable 3-manifolds introduced in 1967 by Waldhausen~\cite{wald:klasse} as a generalization of Seifert fibered 3-manifolds. From the point of view of Thurston's geometrization program, graph manifolds are exactly the compact, orientable 3-manifolds without any hyperbolic piece in their geometric decomposition.

In this article we consider a generalization of the notion of graph manifold that includes some noncompact 3-manifolds. We prove a structure theorem for irreducible open graph manifolds in the form of a canonical `reduced' decomposition along embedded, incompressible 2-tori.
\end{abstract}

\section{Introduction}
In 1967, F. Waldhausen~\cite{wald:klasse} introduced a class of compact, orientable 3-manifolds which he called graph manifolds (`Graphenmannigfaltigkeiten'). These are the compact, orientable 3-manifolds that can be built by gluing together pieces which are $S^1$-bundles over surfaces. In the mid 1970s, W.~Thurston proposed his geometrization conjecture, according to which every compact 3-manifold can be decomposed into pieces possessing a geometric structure modelled on one of the eight geometries $S^3$, $\rr^3$, $\hh^3$ $S^2\times\rr$, $\hh^2\times\rr$, $\widetilde{\mathrm{SL}_2(\rr)}$, $\mathrm{Nil}$ and $\mathrm{Sol}$. From this viewpoint, graph manifolds are exactly the compact, orientable 3-manifolds none of whose pieces are hyperbolic (i.e. modelled on $\hh^3$.) Thurston's conjecture was proven by Perelman in the early aughts~\cite{Per1,Per3,Per2}, see also~\cite{b3mp:book,Cao-Zhu,Kle-Lot,Mor-Tia}. In his proof, graph manifolds appear in the so-called `collapsing case', i.e. roughly speaking, as 3-manifolds admitting sequences of Riemannian metrics with injectivity radius going everywhere to zero and controlled curvature.

From the point of view of Riemannian geometry it is natural to consider open 3-manifolds (i.e. noncompact without boundary) which admit collapsing sequences of complete Riemannian metrics. This motivates the present article.

To discuss our results we give the definition of our generalization of graph manifolds to possibly noncompact 3-manifolds. In this paper all 3-manifolds are smooth, connected, and orientable. 

Let $M$ be a 3-manifold, and $\calF$ be a locally finite collection of pairwise disjoint embedded surfaces in $M$. We denote by $M\setminus \calF$ the manifold obtained from $M$ by removing disjoint open tubular neighborhoods of all members of $\calF$. This manifold is called \bydef{$M$ split along $\calF$}. Its connected components are called \bydef{pieces}.

\begin{defi}
Let $M$ be a 3-manifold without boundary. A \bydef{graph structure} on $M$ is a pair $(\calT,\calS)$ where $\calT$ is a locally finite collection of pairwise disjoint embedded 2-tori in $M$, and $\calS$ is a collection of Seifert fibrations on the pieces of $M\setminus \calT$. If $M$ admits a graph structure, then $M$ is called a \bydef{graph manifold}.\footnote{Our definition of a graph structure is slightly different from Waldhausen, who insists that the pieces be circle bundles. However, since every Seifert manifold can be decomposed into circle bundles, our definition of a graph manifold coincides with his in the compact case.}
\end{defi}

In trying to classify these manifolds we encounter a first difficulty. Recall that a 3-manifold $M$ is called \bydef{irreducible} if every embedded 2-sphere in $M$ bounds a 3-ball, and is called \bydef{prime} if it is irreducible or diffeomorphic to $S^2\times S^1$. The classification of compact graph manifolds is based on the fact that every compact graph manifold is a connected sum of prime ones. This allows to reduce the classification to the irreducible case. However, this fails for open graph manifolds. Indeed, the manifold $M_1$ constructed in~\cite{maillot:examples} is an open graph 3-manifold which cannot be decomposed into prime factors along a locally finite collection of embedded 2-spheres. In this paper we shall only be concerned with \emph{irreducible} open graph manifolds.

To explain the second difficulty, recall that in the compact case, every irreducible graph manifold admits a graph structure whose tori are incompressible, i.e.~$\pi_1$-injective, and such that no two adjacent pieces can be merged to give a larger Seifert fibered piece. This allows to think of graph structures as geometric decompositions in the sense of Thurston. Thus we are led to the following definition.

\begin{defi}
A graph structure $(\calT,\calS)$ on a 3-manifold $M$ is \bydef{reduced} if the following requirements are met:
\begin{enumerate}
\item All members of $\calT$ are incompressible.
\item All Seifert pieces are maximal in the sense that there does not exist a graph structure $(\calT',\calS')$ on $M$ with $\calT'$ a proper subset of $ \calT$.
\end{enumerate}
\end{defi}

While every compact graph manifold has a reduced graph structure,\footnote{This is essentially due to Waldhausen, although his definition of a reduced graph structure is slightly different from ours.} this is not always true of open graph manifolds. Indeed, $\rr^3$ does not contain any incompressible torus, so a reduced graph structure on it would consist of an empty collection of tori and a single Seifert piece. However, the base orbifold in a putative Seifert fibration on $\rr^3$ would have to be a plane, and the only Seifert manifold with base orbifold a plane is $S^1\times\rr^2$. Hence $\rr^3$ is an open graph manifold that does not admit any reduced graph structure.

One could wonder whether $\rr^3$ is the only open graph manifold with this property, but it is not the case. The manifold $M_2$ constructed by Scott and Tucker~\cite{st:exotic} is a fake $S^1\times\rr^2$, i.e.~an irreducible open 3-manifold with infinite cyclic fundamental group which is not diffeomorphic to $S^1\times\rr^2$. Inspection of the construction shows that $M_2$ is a graph manifold. Again $M_2$ does not contain any incompressible torus, and it is not Seifert (if it were, its base orbifold would be a plane with at most one cone point, and again $M_2$ would be diffeomorphic to $S^1\times\rr^2$.) Variations on this construction provide infinitely many such examples.

Note that both $\rr^3$ and $M_2$ have exhaustions by solid tori. We show that this is the only obstruction to the existence of a reduced graph structure.
\begin{theo}\label{theo:main}
Let $M$ be an open, irreducible graph 3-manifold. If $M$ does not admit an exhaustion by solid tori, then $M$ has a reduced graph structure.
\end{theo}

In order to justify that Theorem~\ref{theo:main} provides some kind of `classification' of irreductible graph 3-manifolds that do not have exhaustions by solid tori, we prove a corresponding uniqueness theorem:
\begin{theo}\label{theo:uniqueness}
Let $M$ be an open, irreducible graph 3-manifold. If $(\calT_1,\calS_1)$ and $(\calT_2,\calS_2)$ are two reduced graph structures on $M$, then there is an ambient isotopy of $M$ which carries $\calT_1$ to $\calT_2$.
\end{theo}

The paper is structured as follows. In Section~\ref{sec:prelim}, we gather definitions and results about possibly noncompact Seifert manifolds. Some of them are known, others are straightforward generalizations of results about compact Seifert manifolds; Particularly noteworthy are Theorem~\ref{theo:unique} and Corollary~\ref{corol:unique}, which generalize uniqueness theorems for Seifert fibrations to the noncompact setting. These results are potentially of independent interest.

Section~\ref{sec:removing} deals with the first half of the proof of Theorem~\ref{theo:main}: we explain how to modify an arbitrary graph structure in order to obtain one where all tori are incompressible. The second half of the proof is given in Section~\ref{sec:end}. Finally, Section~\ref{sec:uniqueness} contains the proof of Theorem~\ref{theo:uniqueness}.

\paragraph{Acknowledgments}

I would like to thank Juan Souto and Michel Boileau for useful conversations. This work was partially supported by Agence Nationale de la Recherche through Grant ANR-12-BS01-0004.

\section{Preliminaries}\label{sec:prelim}
In this section we assume known results about compact 2-orbifolds and Seifert 3-manifolds, for which we refer to~\cite{bmp,scott:geom,wald:klasse}, and show how to extend these results to a class of possibly noncompact 2-orbifolds and 3-manifolds, possibly with nonempty boundary. 

We denote by $I$ the interval $[0,1]$. We call open Moebius band (resp.~finite annulus, resp.~half-infinite annulus, resp.~bi-infinite annulus) a surface diffeomorphic to the interior of the Moebius band (resp.~$S^1\times I$, resp.~$S^1\times [0,+\infty)$, resp.~$S^1\times\rr$.) A finite annulus is sometimes just called an \bydef{annulus}.

All 2-orbifolds considered in this article will be connected and locally orientable. The underlying space of such an orbifold is a surface, and its singular points (if any) are cone points.

\begin{defi}
Let $M$ be a 3-manifold. A \bydef{Seifert fibration} on $M$ is a map $p:M\to\oo$ to a 2-orbifold $\oo$ such that $(M,p)$ is a $S^1$-bundle over $\oo$. If $M$ admits a Seifert fibration, it is called a \bydef{Seifert manifold}.
\end{defi}

Let $p:M\to\oo$ be a Seifert fibration. Then $M$ is closed (resp.~open) if and only if $\oo$ is. The map $p$ induces a bijection between the set of connected components of $\bord M$ and that of $\bord\oo$; this bijection takes compact components of $\bord M$ (diffeomorphic to $T^2$) to compact components of $\bord\oo$ (diffeomorphic to $S^1$), and noncompact components of $\bord M$ (diffeomorphic to $S^1\times\rr$) to compact components of $\bord\oo$ (diffeomorphic to $\rr$). The preimages of points of $\oo$ by $p$ are topological circles which are called \bydef{fibers}. The preimages of cone points are called \bydef{singular fibers}.)

Throughout the paper we will always work with Seifert manifolds that do not have any noncompact boundary component. This is because we are primarily interested in decompositions of open graph 3-manifolds into Seifert submanifolds with nonmatching Seifert fibrations, whereas gluing together Seifert 3-manifolds along open annuli always yields matching fibrations. Let $\oo$ be a locally orientable 2-orbifold that does not have any noncompact boundary component. We say that $\oo$ has \bydef{finite type} if there is a compact 2-orbifold $\hat\oo$ such that $\oo$ is diffeomorphic to $\hat\oo \setminus X$ where $X$ is a (possibly empty) union of components of $\bord \hat\oo$. A 2-orbifold $\oo$ has finite type if and only the surface $|\oo|$ has finite type and $\oo$ has finitely many cone points.

\begin{prop}\label{prop:thin}
Let $p:M\to\oo$ be a Seifert fibration. Assume that $M$ is not closed, and $\bord M$ does not have any noncompact component. Then the following properties are equivalent:
\begin{enumerate}
\item The manifold $M$ has virtually abelian fundamental group;
\item The orbifold $\oo$ has virtually abelian fundamental group;
\item The orbifold $\oo$ has finite type and the Euler characteristic of its compactification $\hat\oo$ is nonnegative;
\item The orbifold $\oo$ is a disk (resp.~a plane) with at most one cone point, a disk (resp.~a plane) with exactly two cone points of order $2$, or a Moebius band (resp.~an open Moebius band, resp.~a (finite, half-infinite or bi-infinite) annulus) without cone point;
\item The manifold $M$ is diffeomorphic to $S^1\times D^2$, $T^2\times I$, $K^2\tilde{\times} I$, $S^1\times\rr^2$, $T^2\times\rr$, $K^2\tilde{\times} \rr$, or $T^2\times [0,+\infty)$.
\end{enumerate}
\end{prop}

\begin{proof}
The implication (i) $\implies$ (ii) follows from the fact that the induced homomorphism $p_*:\pi_1M\to\pi_1\oo$ is surjective.

For (ii) $\implies$ (iii) we argue that if $\oo$ has infinite type or $\chi(\hat\oo)<0$, then $\pi_1\oo$ contains a nonabelian free group, hence cannot be virtually abelian.

The implication (iii) $\implies$ (iv) follows from the classification of compact $2$-orbifolds. 

For (iv) $\implies$ (v) we simply do a case by case analysis: if $\oo$ is...

\paragraph{A disk with at most one cone point}
Then $M$ is diffeomorphic to $S^1\times D^2$.

\paragraph{A plane with at most one cone point}
Then $M$ is diffeomorphic to $S^1\times\rr^2$.

\paragraph{A disk with exactly two cone points of order $2$}
Then $M$ is diffeomorphic to $K^2\tilde\times I$.

\paragraph{A plane with exactly two cone points of order $2$}
Then $M$ is diffeomorphic to $K^2\tilde\times \rr$.

\paragraph{A Moebius band with no cone point}
Then $M$ is diffeomorphic to $K^2\tilde\times I$.

\paragraph{An open Moebius band with no cone point}
Then $M$ is diffeomorphic to $K^2\tilde\times \rr$.

\paragraph{A finite annulus without cone point}
Then $M$ is diffeomorphic to $T^2\times I$.

\paragraph{A half-infinite annulus without cone point}
Then $M$ is diffeomorphic to $T^2\times [0,+\infty)$.

\paragraph{A bi-infinite annulus without cone point}
Then $M$ is diffeomorphic to $T^2\times \rr$.

Finally (v) $\implies$ (i) is a simple fundamental group computation: every 3-manifold in the list of (v) has fundamental group isomorphic to $\zz$, $\zz^2$, or $\pi_1K^2$, hence is virtually abelian.
\end{proof}

We note that by characterization~(i) or~(v), this property depends only on the 3-manifold $M$, not the choice of Seifert fibration on it. Hence the following definition makes sense:

\begin{defi}
Let $M$ be a nonclosed Seifert manifold that does not have any noncompact boundary component. We say that $M$ is \bydef{thin} if the properties in Proposition~\ref{prop:thin} are satisfied. Otherwise it is called \bydef{thick}.
\end{defi}

In the sequel, $M$ is a nonclosed Seifert manifold that does not have any noncompact boundary component. If $p:M\to\oo$ is a Seifert fibration, we say that a subset $X$ of $M$ is \bydef{saturated} if it is a union of fibers. We use the terms \bydef{fibered torus} and \bydef{fibered annulus} for a saturated torus and a saturated annulus, respectively. A diffeomorphism between two Seifert manifolds is \bydef{fiber preserving} (with respect to some choice of Seifert fibrations) if it sends fibers to fibers.

The following three results are classical in the case of compact Seifert manifolds, and inspection of the proofs show that they extend readily to the noncompact case.

\begin{lem}\label{lem:vert}
Let $p:M\to\oo$ be a Seifert fibration. Let $T$ be an incompressible torus in $M$. Then $T$ is isotopic to a fibered torus. 
\end{lem}

\begin{lem}\label{lem:ann vert}
Let $p:M\to\oo$ be a Seifert fibration. Let $A$ be an essential annulus in $M$.
\begin{enumerate}
\item If $M$ is thick, then $A$ is isotopic to some fibered annulus.
\item If $M$ is diffeomorphic to $K^2\tilde{\times} I$, then there exists a Seifert fibration on $M$ such that $A$ is isotopic to some fibered annulus.
\end{enumerate}
\end{lem}

\begin{prop}\label{prop:extend}
Let $p:M\to\oo$ be a Seifert fibration. Let $N$ be a 3-manifold obtained by gluing a solid torus $V$ to $M$ along some boundary component $T$ of $M$. If the generic fiber of $T$ does not bound a meridian disk in $V$, then the Seifert fibration $p$ extends to some Seifert fibration on $N$. 
\end{prop}

\begin{theo}\label{theo:unique}
Let $p_1:M_1\to \oo_1$ and $p_2:M_2\to \oo_2$ be Seifert fibrations. Assume that $M_1$ is thick. Let $\phi:M_1\to M_2$ be a diffeomorphism. Then there is a fiber preserving diffeomorphism $\psi:M_1\to M_2$ which is isotopic to $\phi$.
\end{theo}

\begin{proof}
If $M_1$ (and therefore $M_2$) is compact, this is due to Waldhausen~\cite[Satz 2.9]{wald:klasse}.

Therefore we assume that $M_1$ is noncompact. By characterization~(iii) of thin manifolds, we see that $M_1$ has an exhaustion by compact thick submanifolds which are saturated for $p_1$ and all of whose boundary components are incompressible in $M_1$. Let $\{U_n\}_{n\in\nn}$ be such an exhaustion. Without loss of generality we assume that for every $n$, each component of $U_{n+1}\setminus \Int{U_n}$ is either thick or diffeomorphic to $T^2\times I$.

Let $T_{0,1},\ldots,T_{0,p}$ be the boundary components of $U_0$. For each $i$ we apply Lemma~\ref{lem:vert} and obtain an isotopy between $\phi(T_{0,i})$ and some fibered torus $T'_{0,i}$ in $M_2$. We then extend these isotopies to an isotopy between $\phi$ and some diffeomorphism $\phi'_0:M_1\to M_2$ such that $\phi'_0(U_0)$ is the saturated (compact) submanifold of $M_2$ bounded by the $T'_{0,i}$'s. Applying the compact case of Theorem~\ref{theo:unique} to $U_0$, we modify $\phi'_0$ by an isotopy to get a diffeomorphism $\phi_0$ with the same property and such that moreover the restriction of $\phi_0$ to $U_0$ sends fiber to fiber. We can ensure that these two isotopies have support in some small neighborhood $U'_0$ of $U_0$ contained in $U_1$.

We now work on $U_1$. Let $T_{1,1},\ldots,T_{1,p}$ be the boundary components of $U_1$. For each $i$ we apply Lemma~\ref{lem:vert} and obtain an isotopy between $\phi(T_{1,i})$ and some fibered torus $T'_{1,i}$ in $M_2$. We then extend these isotopies to an isotopy between $\phi_0$ and some diffeomorphism $\phi'_1:M_1\to M_2$ such that $\phi'_1(U_1)$ is the saturated submanifold of $M_2$ bounded by the $T'_{1,i}$'s. We choose such an isotopy with support in some neighborhood $U'_1$ of $U_1$ contained in $U_2$. Then by an isotopy again supported in $U'_1$ and constant on $U_0$ we obtain a diffeomorphism $\phi_1$ with the same property and whose restriction to $U_1$ is fiber preserving.

We iterate this construction, getting an infinite sequence of isotopies such that any compact subset of $M_1$ is in the support of only finitely many (actually, at most four) of them. Piecing together these isotopies we obtain an isotopy from $\phi$ to some fiber preserving diffeomorphism $\psi$.
\end{proof}

\begin{corol}\label{corol:unique}
Let $M$ be a thick Seifert manifold. Then there is only one Seifert fibration on $M$ up to isotopy.
\end{corol}

We end this section by a lemma which will be useful later.
\begin{lem}\label{lem:thick}
Let $Y$ be a nonclosed Seifert manifold. Let $X\subset Y$ be a thick Seifert manifold such that every component of $\bord X$ is incompressible in $Y$. Then $Y$ is thick and every Seifert fibration on $X$ extends to a Seifert fibration of $Y$.
\end{lem}

\begin{proof}
It follows from incompressibility of $\bord X$ that $\pi_1X$ injects into $\pi_1Y$. Hence using characterization~(i), thickness of $X$ implies thickness of $Y$.

Let $p$ be a Seifert fibration on $X$. To construct an extension of $p$ to $Y$, we start with an arbitrary Seifert fibration $q$ on $Y$. Let $T_1,\ldots,T_n$ be the boundary components of $X$. By Lemma~\ref{lem:vert}, each $T_i$ is isotopic to some torus $T'_i$ that is a union of fibers of $q$. As in the proof of~\ref{theo:unique}, we can find an ambient isotopy of $Y$ which brings $T_i$ to $T'_i$ for every $i$. The reverse isotopy will transform the Seifert fibration $q$ into a Seifert fibration $q'$ of $Y$ for which $\bord X$, hence $X$, is saturated.

By Corollary~\ref{corol:unique}, the restriction of $q'$ to $X$ is isotopic (in $X$) to $p$. This isotopy can be extended to an ambient isotopy in $Y$, yielding a Seifert fibration on $Y$ whose restriction to $X$ is equal to $p$.
\end{proof}

\section{Removing compressible tori}\label{sec:removing}
The goal of this section is to prove the following proposition:
\begin{prop}\label{prop:incomp}
Let $M$ be an open, irreducible graph 3-manifold which does not admit any exhaustion by solid tori. Then $M$ has a graph structure $(\calT_0,\calS_0)$ such that every member of $\calT_0$ is incompressible.
\end{prop}

We shall use the following lemma~\cite[Lemma A.3.1]{b3mp:book}.

\begin{lem}\label{lem:comp torus}
Let $T$ be a torus embedded in $M$. If $T$ is compressible, then $T$ bounds a compact submanifold $V\subset M$. Moreover, if $V$ is not a solid torus, then  $V$ is contained in a submanifold $B$ of $M$ which is diffeomorphic to the 3-ball, and $V$ can be replaced by a solid torus without changing the diffeomorphism type of $M$.
\end{lem}

\begin{proof}[Proof of Proposition~\ref{prop:incomp}]
Let $M$ be an open, irreducible graph 3-manifold which does not admit any exhaustion by solid tori. In particular, $M$ is neither diffeomorphic to $\rr^3$ nor to $S^1\times\rr^2$. We start with an arbitrary graph structure $(\calT,\calS)$ on $M$, and denote by $\calT'$ the subset of $\calT$ consisting of compressible tori.

To every $T\in\calT'$ we associate the compact submanifold $V(T)$ given by Lemma~\ref{lem:comp torus}. (It is unique since $M$ is noncompact.) For every $T_1,T_2\in\calT'$ we have $V(T_1)\subset V(T_2)$, $V(T_2)\subset V(T_1)$ or $V(T_1)\cap V(T_2)=\emptyset$.  The set $\calV=\{V(T) \mid T\in\calT'\}$ is partially ordered by inclusion.

\begin{lem}\label{lem:maximal}
Every element of $\calV$ is contained in a maximal element.
\end{lem}

\begin{proof}
Assuming the contrary, we get an infinite sequence $T_1,\ldots,T_k,\ldots$ of members of $\calT'$ such that the sequence $(V(T_k))_{k\ge 1}$ is increasing. Since $\calT'$ is locally finite, this is an exhaustion of $M$. As $M$ does not admit any exhaustion by solid tori, infinitely many $V(T_k)$'s must fail to be solid tori. Then by Lemma~\ref{lem:comp torus}, $M$ has an exhaustion by 3-balls. This implies that $M$ is diffeomorphic to $\rr^3$, which contradicts our hypothesis.
\end{proof}

Having proved Lemma~\ref{lem:maximal}, we continue the proof of Proposition~\ref{prop:incomp}. We remove from $\cal T$ every member $T$ of $\calT'$ such that $V(T)$ is not maximal, getting a subcollection $\calT_1$ of $\calT$ such that every piece of $M\setminus \calT_1$ is a piece of $M\setminus \calT$,  a solid torus, or contained in a 3-ball. In the latter case, using Lemma~\ref{lem:comp torus} we replace this piece by a solid torus. Then there exists $\calS_1$ such that $(\calT_1,\calS_1)$ is a graph structure on $M$.

\begin{lem}\label{lem:adjacent thick}
Let $X$ be a piece of $M\setminus \calT_1$ which is also a piece of $M\setminus \calT$ and is adjacent to at least one solid torus $V(T)$. Then $X$ is thick.
\end{lem}

\begin{proof}
By hypothesis $X$ has nonempty boundary; since $M$ is open, $X$ is noncompact or has at least two boundary components. Therefore, if $X$ is thin, then $X$ is diffeomorphic to $T^2\times I$ or $T^2\times [0,+\infty)$. In the first case, $\bord X$ is a union of two tori $T_1,T_2$, with $T_1$ bounding a solid tori $V(T_1)$ whose interior is disjoint from $X$. Then $X\cup V(T_1)$ is a solid torus, contradicting the maximality of $V(T_1)$. In the second case, $M$ is diffeomorphic to $S^1\times \rr^2$, giving again a contradiction.
\end{proof}

Let $X$ be a piece of $M\setminus \calT_1$ satisfying the hypothesis of Lemma~\ref{lem:adjacent thick}. By Corollary~\ref{corol:unique}, $X$ has a unique Seifert fibration up to isotopy, which we denote by $p:X\to\oo$. Let $T\in\calT'$ be a boundary component of $X$. By Proposition~\ref{prop:extend}, either $p$ extends to $V(T)$, or its generic fiber bounds a meridian disk of $V(T)$. Let $Y$ be the union of $X$ and all the $V(T)$'s where $T$ is as above and the Seifert fibration of $X$ extends to $V(T)$. Then $Y$ is Seifert fibered. We remove from $\calT'$ all those tori and get a graph structure with $Y$ a piece. If $Y$ is still adjacent to a solid torus, then arguing as above, we see that $Y$ is thick.

We perform this modification on all pieces $X$ that are adjacent to at least one solid torus. We obtain a graph structure $(\calT_0,\calS_0)$ with the property that for every piece $Y$ of this graph structure, if there is a solid torus $V$ adjacent to $Y$ then $Y$ is thick, and the fiber of the Seifert fibration on $Y$ bounds a meridian disk on each such solid torus.

In order to complete the proof of Proposition~\ref{prop:incomp}, we argue by contradiction and assume that there remains at least one solid torus $V$ among the pieces of $(\calT_0,\calS_0)$. We let $Y$ be the piece which is adjacent to $V$ and $p:Y\to\oo$ be its Seifert fibration. We denote by $T$ the boundary of $V$.

A properly embedded arc $\alpha\subset\oo$ is called \bydef{admissible} if it has both endpoints in $p(T)$. To any admissible arc $\alpha$ we associate an embedded 2-sphere $S(\alpha)$ obtained by capping the annulus $p^{-1}(\alpha)$ by two disjoint meridian disks of $V$.

If $|\oo|$ is not planar, then we can find a nonseparating admissible arc $\alpha\subset\oo$. The 2-sphere $S(\alpha)$ is nonseparating, which contradicts the irreducibility of $M$. Hence $|\oo|$ is planar.

If there is another solid torus $V'$ adjacent to $Y$, then we let $\alpha$ be a properly embedded arc in $|\oo|$ with one endpoint in $p(T)$ and the other in $p(\bord V')$. Again we can glue two disks to the annulus $p^{-1}(\alpha)$, getting a nonseparating 2-sphere. Hence every boundary component of $Y$ other than $T$ (if any) is incompressible.

Suppose that $\oo$ has more than one end, and let $l$ be a properly embedded line connecting two different ends. Then there is an admissible arc $\alpha$  intersecting $l$ transversely in one point. The 2-sphere $S(\alpha)$ separates two ends of $M$, again a contradiction. Hence $\oo$ has at most one end. We are left with two cases.

\paragraph{Case 1} $\oo$  is compact.

Then $\bord Y$ must have at least one incompressible boundary component. If there are two such components, say $T_1,T_2$, then let $\alpha,\beta$ be properly embeded arcs in $|\oo|$ such that $\alpha$ is admissible, $\beta$ has an endpoint in $p(T_1)$ and one in $p(T_2)$, and $\alpha,\beta$ intersect transversely in one point. Since $M$ is irreducible, $S(\alpha)$ bounds a 3-ball $B$. From the existence of the arc $\beta$ we deduce that $T_1,T_2$ belong to different components of $M\setminus S$. Hence one of them is contained in $B$, which is impossible since they are both incompressible. Therefore, $|\oo|$ is a compact planar surface with exactly two boundary components, i.e.~an annulus. We denote by $T'$ the incompressible boundary component of $Y$.

Since $Y$ is thick, $\oo$ must have at least one cone point, say $x$. Let $f$ be the singular fiber above $x$. Let $\alpha,\beta$ be arcs in $|\oo|$ such that $\alpha$ is admissible, $\beta$ connects $x$ to $p(T')$, and $\alpha,\beta$ intersect transversely in one point. Arguing as above, we see that the 2-sphere $S(\alpha)$  separates $f$ from $T'$. Since $T'$ is incompressible, the 3-ball $B$ bounded by $S(\alpha)$ must contain $f$. This implies that $f$ is null-homotopic in the manifold $Z:=Y\cup V$. Computing $\pi_1Z$ using the van Kampen theorem and noting that only the regular fiber, which is a proper power of $f$, gets killed by adding $V$, we reach a contradiction.

\paragraph{Case 2} $\oo$ is one-ended.

This case is analogous to Case 1. First we prove that $\bord Y=T$ using a line connecting the image by $p$ of the putative incompressible component of $\bord Y$ which plays the role of $\beta$ in the compact case. Thus $|\oo|$ is diffeomorphic to $S^1\times [0,+\infty)$. We then prove that $\oo$ has no cone point, which implies that $M$ is diffeomorphic to $S^1\times\rr^2$. This contradiction finishes the proof of Proposition~\ref{prop:incomp}.
\end{proof}

\section{End of the proof of Theorem~\ref{theo:main}}\label{sec:end}
Let $M$ be an open, irreducible $3$-manifold endowed with a graph structure $(\calT_0,\calS_0)$ such that every member of $\calT_0$ is incompressible. If $M$ is Seifert fibered, there is nothing to prove, so we assume it is not. Thus every piece is either thick or diffeomorphic to $T^2\times I$, $K^2\tilde\times I$, or $T^2\times [0,+\infty)$.

Our first task is to get rid of the pieces diffeomorphic to $T^2\times I$. With this goal in mind, we let $G$ be the dual graph of $\calT_0$, and $H$ be the subgraph consisting of the vertices corresponding to pieces diffeomorphic to $T^2\times I$, together with the edges that connect them. If $v$ is a vertex of $H$ we denote by $X_v$ the corresponding piece.

Let $J$ be a connected component of $H$, and let $X$ be the union of the pieces corresponding to $J$. If $J$ is finite, then $X$ is diffeomorphic to $T^2\times I$ and adjacent to either one or two pieces. In the former case, we remove $X$ and glue the adjacent piece to itself. In the latter, we remove $X$ and glue the two adjacent pieces together.

If $J$ is infinite, then it is a half-line. Indeed, if it were a line, then $M$ would be diffeomorphic to $T^2\times \rr$, hence Seifert fibered. As a consequence, $X$ is diffeomorphic to $T^2\times [0,\infty)$ and adjacent to a unique piece $Y$. The Seifert fibration on $Y$ induced by $\calS_0$ extends to $Y\cup X$, so we may merge $Y$ with $X$.

Performing this modification on each connected component of $H$, we obtain a graph structure $(\calT_1,\calS_1)$ on $M$ such that every member of $\calT_1$ is incompressible, and no piece of $(\calT_1,\calS_1)$ is diffeomorphic to $T^2\times I$.

If there are pieces diffeomorphic to $T^2\times [0,+\infty)$, we can merge them with the adjacent pieces as explained above. This yields a graph structure $(\calT_2,\calS_2)$ with all the properties of $(\calT_1,\calS_1)$, and in addition having no piece diffeomorphic to $T^2\times [0,+\infty)$. At this stage, all pieces are either thick or diffeomorphic to $K^2\tilde\times I$.

We need to choose a Seifert fibration on each piece diffeomorphic to $K^2\tilde\times I$. To this end, we note that for every such piece $X$ there is a unique piece $Y$ adjacent to $X$, and $Y$ is thick (otherwise $M$ would be closed.) If the (unique up to isotopy) Seifert fibration on $Y$ can be extended to $X$, we merge $Y$ and $X$, getting a thick piece $Y'$. Otherwise we fix an arbitrary Seifert fibration on $X$. Let $(\calT_3,\calS_3)$ be the result of this construction.

Let $T$ be a member of $\calT_3$ that is not adjacent to a thin piece. Let $T',T''$ be the two boundary components  of a tubular neighborhood $U$ of $T$. If the circle fibrations on $T',T''$ induced by $\calS_3$ match, then we remove $T$ and adjust $\calS_3$ on $U$. Note that there are two cases: $T',T''$ may belong to a single piece of $(\calT_3,\calS_3)$ or two different pieces. In the latter case, those two pieces are merged. By Lemma~\ref{lem:thick}, the new pieces produced by this construction are thick, so we can iterate it. This process (which may be finite or infinite) yields a reduced graph structure $(\calT_4,\calS_4)$. Thus the proof of Theorem~\ref{theo:main} is complete.

\section{Uniqueness}\label{sec:uniqueness}
We turn to the proof of Theorem~\ref{theo:uniqueness}.

We first show that it suffices to prove the following claim:
\begin{claim}
Every member $T$ of $\calT_1$ is isotopic to some member $T'$ of $\calT_2$.
\end{claim}

Indeed, if the claim is true, then for each $T$ the corresponding $T'$ is necessarily unique (otherwise there would be two distinct members of $\calT_2$ isotopic to each other, which is impossible since $\calT_2$ is reduced). By symmetry the corresponding statement with the roles of $\calT_1$ and $\calT_2$ reversed also holds. Thus there is a bijective correspondence between members of $\calT_1$ and members of $\calT_2$ which to every member $T$ of $\calT_1$ associates a member of $\calT_2$ isotopic to it. It follows that there is an ambient isotopy carrying $\calT_1$ to $\calT_2$.

Let us prove the claim. We will freely modify $\calT_1$ and $\calT_2$ by isotopies in order to achieve various properties of their intersection. At each step, we may have to perform infinitely many isotopies, but they have disjoint supports, and each compact set of $M$ is concerned by only finitely many of these isotopies.

\paragraph{Step 1} By a standard general position argument, we can ensure that after some isotopies (as explained in the previous paragraph) all members of $\calT_1$ and $\calT_2$ intersect transversely along (locally finitely many) simple closed curves.

\smallskip

Let $T_1\in\calT_1$ and $T_2\in \calT_2$ and let $\gamma$ be a curve in $T_1\cap T_2$. Since both $T_1$ and $T_2$ are incompressible, the curve $\gamma$ is essential in $T_1$ if and only if it is essential in $T_2$. Thus we may refer to intersection curves as simply `essential' or `inessential'.

\paragraph{Step 2}  After some isotopies we can ensure that there are no inessential intersection curves.

\smallskip

Let $\gamma$ be an inessential curve in $T_1\cap T_2$ with $T_1\in\calT_1$ and $T_2\in \calT_2$ which is innermost in $T_2$. Then there exist disks $D_1\subset T_1$ and $D_2\subset T_2$ such that $\bord D_1=\bord D_2=\gamma$ and $\Int D_2 \cap \bigcup \calT_1=\emptyset$.

Since $M$ is irreducible, the sphere $D_1\cup D_2$ bounds a 3-ball. Hence there is an isotopy of $T_1$ which allows us to get rid of $\gamma$.

\paragraph{Step 3} After some isotopies we can ensure that $\bigcup \calT_1$ and $\bigcup \calT_2$ are disjoint.

\smallskip 

Let $T\in\calT_1$ be a torus which is not disjoint from $\bigcup\calT_2$ and let $\gamma\subset T$ be an intersection curve. We distinguish two cases according to whether $T$ is adjacent to one or two pieces of $M\setminus \calT_1$.

First assume that $T$ is adjacent to two pieces $X_1,X_2$ of $M\setminus \calT_1$. We let $T_1$ (resp.~$T_2$) be the boundary component of $X_1$ (resp.~$X_2$) coming from cutting $M$ along $T$, and we let $\gamma_1$ (resp.~$\gamma_2$) be a curve parallel to $\gamma$ in $T_1$ (resp.~$T_2$). Looking at the trace of $\calT_2$ on $X_1$ and $X_2$, we get for each $i\in\{1,2\}$ a properly embedded annulus $A_i\subset X_i$ one of whose boundary components is $\gamma_i$. Since $\gamma$ is an essential curve on $T$, the annuli $A_i$ are incompressible.  If both $A_1$ and $A_2$ are essential, then by Lemma~\ref{lem:ann vert} we can choose Seifert fibrations on $X_1$ and $X_2$ such that each $A_i$ is fibered. This implies that these Seifert fibrations match along $T$, which is impossible since $\calT_1$ is reduced. Hence at least one $A_i$ is boundary parallel, and by an innermost argument we can reduce the number of intersection curves on $T$.

If $T$ is adjacent to only one piece $X$ of $M\setminus \calT_1$ the argument is similar. Let $T_+$ and $T_-$ be the two boundary components of $X$ coming from cutting $M$ along $T$, and let $\gamma_+\subset T_+$ and $\gamma_-\subset T_-$ be curves parallel to $\gamma$. Since $X$ has at least two boundary components, it cannot be diffeomorphic to $K^2\tilde{\times} I$, so it is thick. Thus looking at the trace of $\calT_2$ on $X$ yields either two incompressible properly embedded annuli $A_+,A_-\subset X$ one of whose boundary component is $\gamma_+$ and $\gamma_-$ respectively, or one incompressible properly embedded annulus $A\subset X$ whose boundary is $\gamma^+\cup \gamma^-$. Again using Lemma~~\ref{lem:ann vert} we see that the latter case is excluded, and in the former case at least one of $A_+$ and $A_-$ is inessential, allowing to reduce the number of intersection curves on $T$.

\paragraph{Step 4} After some isotopies, all members of $\calT_1$ (resp.~$\calT_2)$ are fibered with respect to $\calS_2$ (resp.~$\calS_1$).

\smallskip 

This follows from repeated applications of Lemma~\ref{lem:vert}, first to members of $\calT_1$, then to members of $\calT_2$.

\paragraph{Step 5} End of proof of the claim.

\smallskip

Arguing by contradiction, let $T\in\calT_1$ which is not isotopic to any member of $\calT_2$.

\paragraph{Case 1} $T$ is adjacent to two pieces $X_1,X_2$ of $M\setminus (\calT_1\cup\calT_2)$. 

First we observe that neither $X_1$ nor $X_2$ can be diffeomorphic to $T^2\times I$ (otherwise either $\calT_1$ would not reduced, or $T$ would be isotopic to some member of $\calT_2$.)

Let $Y$ be the piece of $M\setminus \calT_2$ containing $X_1$ and $X_2$. If $Y$ were diffeomorphic to $K^2\tilde{\times} I$, then $T$ would be isotopic to $\bord Y$, contradicting our hypothesis. Hence $Y$ is thick. Let $p$ be the Seifert fibration on $Y$ coming from $\calS_2$. By assumption, $T$ is fibered with respect to $p$. For each $i\in\{1,2\}$, let $p_i$ be the restriction of $p$ to $X_i$.

Suppose that both $X_1$ and $X_2$ are thick. We now have two subcases. Assume first that $T$ is adjacent to two pieces of $M\setminus \calT_1$, say $Z_1,Z_2$ with $X_1\subset Y\cap Z_1$ and $X_2\subset Y\cap Z_2$. By Lemma~\ref{lem:thick} applied twice, for each $i$ we deduce that $Z_i$ is thick, and carries a Seifert fibration whose restriction to $X_i$ is $p_i$. This contradicts the hypothesis that $\calT_1$ is reduced. If by contrast $T$ is adjacent to only one piece $Z$ of $M\setminus \calT_1$, then $X_1\cup X_2\subset Z$; by Lemma~\ref{lem:thick}, $Z$ is thick and carries a Seifert fibration whose restriction to each $X_i$ is $p_i$. Again this contradicts the hypothesis that $\calT_1$ is reduced. 

Suppose now that one of $X_1$ and $X_2$, say $X_1$, is diffeomorphic to $K^2\tilde{\times} I$. Since $M$ is open, $X_2$ is thick. Then $\bord X_1$ is connected, so it is equal to $T$, and there are two components of $M\setminus \calT_1$ adjacent to $T$: one is $X_1$, and the other one, which we call $Z_2$, contains $X_2$, hence is thick. Arguing as above we show that some Seifert fibration on $Z_2$ can be extended over $X_1$, a contradiction.

\paragraph{Case 2}

Assume now that $T$ is adjacent to only one piece $X$ of $M\setminus (\calT_1\cup\calT_2)$. Then $X$ has at least three boundary components, so it is thick. Let $Y$ be the piece of $M\setminus \calT_2$ containing $X$. By Lemma~\ref{lem:thick}, $Y$ is thick, and its unique Seifert fibration restricts to some Seifert fibration $p$ on $X$ which matches itself across $T$. Letting $Z$ be the only piece of $M\setminus \calT_1$ adjacent to $T$ and arguing as above we get a contradiction. This completes the proof of the Claim, hence that of Theorem~\ref{theo:uniqueness}.

\bibliographystyle{abbrv}
\bibliography{graph}

Institut Montpelli\'erain Alexander Grothendieck,
CNRS - Universit\'e de Montpellier.\\ 
\texttt{sylvain.maillot@umontpellier.fr}

\end{document}